\documentclass[reqno,12pt]{amsart}

\usepackage{amsmath, amsthm, amsfonts, amssymb}
\input{mathrsfs.sty}

\usepackage{amsmath}
\usepackage{amsfonts}
\usepackage{amssymb}
\usepackage{amsthm}
\usepackage{amstext}
\newcommand{\Ker}{\operatorname{Ker}}
\newcommand{\id}{\operatorname{id}}

\newcommand{\var}{\operatorname{var}}

\newcommand{\Der}{\operatorname{Der}}
\newcommand{\Inn}{\operatorname{Inn}}
\newcommand{\Out}{\operatorname{Out}}

   \theoremstyle{plain}%default
   \newtheorem{thm}{Theorem}[section]
   \newtheorem{prop}[thm]{Proposition}
   \newtheorem{lem}[thm]{Lemma}
   \newtheorem{cor}[thm]{Corollary}
   \theoremstyle{definition}

   \theoremstyle{remark}

 \numberwithin{equation}{section}

\author{V. Manuilov}

\date{}

\address{Moscow Center for Fundamental and Applied Mathematics, Moscow State University,
Leninskie Gory 1, Moscow, 
119991, Russia}

\email{manuilov@mech.math.msu.su}

\title[]{On Hochschild homology of uniform Roe algebras with coefficients in uniform Roe bimodules}


\begin{document}

\begin{abstract}
It was shown recently by M. Lorentz and R. Willett that all bounded derivations of the uniform Roe algebras of metric spaces of bounded geometry are inner. Here we calculate the space of outer derivations of the uniform Roe algebras with coefficients in uniform Roe bimodules related to various metrics on the two copies of the given space. We also give some results on the higher Hochschild cohomology with coefficients in uniform Roe bimodules. 

\end{abstract}

\maketitle

\section*{Introduction}

Uniform Roe algebras play an important role in coarse geometry and in noncommutative geometry. Along with uniform Roe algebras, it is interesting to consider uniform Roe bimodules introduced in \cite{M-RJMP} and related to various metrics on the two copies of the same space. It was shown recently in \cite{Lorentz-Willett} that all bounded derivations of the uniform Roe algebra over any discrete metric space of bounded geometry are inner. 
Here we study bounded Hochschild cohomology of the uniform Roe algebras with coefficients in uniform Roe bimodules. Our main result is calculation of the first bounded Hochschild cohomology group with coefficients in uniform Roe bimodules for discrete metric spaces of bounded geometry. We also prove that any odd bounded Hochschild cohomology is non-trivial, and show that the K\"unneth theorem does not hold for Hochschild cohomology.

\section{Definitions and notation}

Let $X$ be a countable discrete metric space with a metric $d_X$. Let $l^2(X)$ denote the Hilbert space of square-summable functions on $X$ with the fixed orthonormal basis consisting of the delta functions $\delta_x$, $x\in X$, annd let $\mathbb B(l^2(X))$ (resp., $\mathbb K(l^2(X))$) denote the $C^*$-algebra of all bounded (resp., compact) operators on $l^2(X)$. For an operator $a\in\mathbb B(l^2(X))$, we denote its matrix elements by $a_{xy}$, $x,y\in X$, where $a_{x,y}=\langle \delta_x,a\delta_y\rangle$. An operator $a\in\mathbb B(l^2(X))$ has propagation not exceeding $r$ if $a_{xy}=0$ whenever $d_X(x,y)>r$. The norm closure $C_u^*(X)$ of the set of all bounded operators of finite propagation is a $C^*$-algebra called the uniform Roe algebra of $X$.

Recall that two metrics, $d_X$ and $b_X$, on $X$ are coarsely equivalent if there exists a homeomorphism $\varphi$ on $[0,\infty)$ such that $d_X(x,y)\leq\varphi(b_X(x,y))$ and $b_X(x,y)\leq\varphi(d_X(x,y))$ for any $x,y\in X$. Clearly, the uniform Roe algebra of $X$ does not depend on the metric within its coarse equivalence class.

A metric space $X$ is of {\it bounded geometry} if for any $r>0$, the number of points in any ball of radius $r$ is bounded by some constant depending on $r$. Bounded geometry condition also does not depend on the metric within its coarse equivalence class.

Let $(X,d_X)$ and $(Y,d_Y)$ be discrete metric spaces, and let $\rho$ be a metric on $X\sqcup Y$ such that $\rho|_X$ and $\rho|_Y$ are coarsely equivalent to $d_X$ and $d_Y$, respectively. We call such metric $\rho$ {\it compatible} with $d_X$ and $d_Y$. Let $M_\rho\subset\mathbb B(l^2(Y),l^2(X))$ be the norm closure of the set of all bounded operators from $l^2(Y)$ to $l^2(X)$ of finite propagation with respect to $\rho$. Clearly, $M_\rho$ is a right $C_u^*(Y)$-module and a left $C_u^*(X)$-module, and the two module structures commute. When $(Y,d_Y)=(X,d_X)$, $M_\rho$ is a $C_u^*(X)$-bimodule. We call it the {\it uniform Roe bimodule} determined by (the coarse equivalence class of) $\rho$. Note that this definition slightly differs from that in \cite{M-RJMP}.

In what follows, it may be useful to distinguish between the two copies of $X$ in $X\sqcup X=X\times\{0,1\}$. For shortness' sake, we write $x$ for $(x,0)$, and $x'$ for $(x,1)$. 

Let $\mathcal D=(D_n)_{n\in\mathbb N}$ be a sequence of subsets in $X$ such that 
\begin{itemize}
\item[(1)]
$D_n$ is non-empty for some $n\in\mathbb N$;
\item[(2)]
there exists $r>0$ such that $N_r(D_n)\subset D_{n+1}$ for any $n\in\mathbb N$, where $N_r(D)$ denotes the $r$-neighborhood of the set $D\subset X$.
\end{itemize} 
We call such sequences {\it expanding sequences}. Note that $\cup_{n\in\mathbb N}D_n=X$. Two expanding sequences, $\mathcal D$ and $\mathcal D'$, are equivalent if there exists a monotone self-map $\varphi$ of $\mathbb N$ such that $D_n\subset D'_{\varphi(n)}$ and $D'_n\subset D_{\varphi(n)}$ for any $n\in\mathbb N$.

Let $\rho$ be a metric on $X\sqcup X$ compatible with $d_X$. Set $D_n=\{x\in X:\rho(x,x')\leq n\}$. Then clearly $N_{1/2}(D_n)\subset D_{n+1}$ for any $n\in\mathbb N$, where the distance is with respect to the metric $\rho|_X$. Replacing it by the metrix $d_X$, we can find $r>0$ such that the $r$-neighborhood with respect to $d_X$ is contained in the $1/2$-neighborhood with respect to the metric $\rho|_X$, hence $\mathcal D=(D_n)_{n\in\mathbb N}$ is an expanding sequence. For a metric $\rho'$ coarsely equivalent to $\rho$ (and, therefore, compatible with $d_X$), let $\mathcal D'=(D'_n)_{n\in\mathbb N}$ be its expanding sequence. Then $\mathcal D$ and $\mathcal D'$ are equivalent. 

Let $A$ be a $C^*$-algebra, and let $M$ be a Banach $A$-$A$-bimodule. Let $C^n(A,M)$ denote the space of all bounded multilinear maps from $A^n$ to $M$. The Hochschild differential $\delta_n:C^n(A,M)\to C^{n+1}(A,M)$ is defined by
\begin{eqnarray*}
\delta_n(\varphi)(a_1,\ldots,a_{n+1})&=&a_1\varphi(a_2,\ldots,a_{n+1})\\
&+&\sum_{i=1}^n (-1)^i\varphi(a_1,\ldots,a_ia_{i+1},\ldots,a_{n+1})\\
&+&(-1)^{n+1}\varphi(a_1,\ldots,a_n)a_{n+1},
\end{eqnarray*}
$a_1,\ldots,a_{n+1}\in A$, $\varphi\in C^n(A,M)$. The Hochschild cohomology groups $HH^n(A,M)$ are standardly defined by $HH^n(A,M)=\Ker\delta_n/\operatorname{Im}\delta_{n-1}$. 

A bounded map $d:A\to M$ is a {\it derivation} if $d(ab)=d(a)b+ad(b)$ for any $a,b\in A$. A derivation $d$ is {\it inner} if there exists $m\in M$ such that $d(a)=[a,m]=am-ma$ for any $a\in A$. The set of all (resp., of all inner) derivations is denoted by $\Der(A,M)$ (resp., by $\Inn(A,M)$). The quotient $\Out(A,M)=\Der(A,M)/\Inn(A,M)$ is the space of the outer derivations.  
The group $HH^1(A,M)$ is canonically identified with $\Out(A,M)$. 

Our reference for Hochschild cohomology is \cite{Witherspoon}. Necessary modifications for Banach algebras and Banach bimodules can be found in \cite{Helemskii}.
We shall use the following result from \cite{Helemskii}, Part 3, Theorem 4.17. A proof for derivations when $M=A$, based on \cite{Kaplansky} can be found in \cite{Chernoff}.

\begin{thm}\label{Hel}
Let $H$ be a Hilbert space, Let $A$ be a $C^*$-algebra, $\mathcal K(H)\subset A\subset\mathbb B(H)$, and let $M=\mathbb B(H)$. Then $HH^n(A,M)=0$ for any $n\in\mathbb N$. In particular, all derivations of $A$ taking values in $M$ are inner.

\end{thm}

\section{Some $C^*$-algebras of functions on $X$}

Let $C_b(X)$ denote the algebra of all bounded continuous functions on $X$.
Given an expanding sequence $\mathcal D$ in $X$, let us define two $C^*$-algebras of functions on $X$ related to $\mathcal D$, $C_0(X,\mathcal D)$ and $C_h(X,\mathcal D)$. 

1. Let $f\in C_b(X)$. We say that $f$ vanishes far from $\mathcal D$ if 
$$
\lim_{n\to\infty}\sup_{x\in X\setminus D_n}|f(x)|=0.
$$
Define the $C^*$-algebra of all bounded functions vanishing far from $\mathcal D$ by $C_0(X,\mathcal D)$. If $\mathcal D'$ is equivalent to $\mathcal D$ then we get the same algebra. 

2. Let $f\in C_b(X)$ be a bounded function on $X$. For $r>0$, set 
$$
\var_{x,r}f=\max_{y\in X:d_X(x,y\leq r)}|f(y)-f(x)|. 
$$
We say that $f\in C_b(X)$ satisfies the Higson condition with respect to $\mathcal D$ if 
$$
\lim_{n\to\infty}\sup_{x\in X\setminus D_n}\var_{x,r}f=0
$$   
for any $r>0$. Clearly, this condition depends only on the coarse equivalence classs of $d_X$, but not on $d_X$ itself. It is also depends only on the equivalence class of expanding sequences. Define the Higson algebra $C_h(X,\mathcal D)$ with respect to $\mathcal D$ as the algebra of all bounded functions on $X$ satisfying the Higson condition. Roughly speaking, $C_h(X,\mathcal D)$ consists of bounded functions with variation vanishing far from $\mathcal D$. Note that if all $D_n$, $n\in\mathbb N$, are bounded then we get the classical Higson algebra \cite{Roe}, i.e. the algebra of continuous functions on the Higson compactification of $X$.

For shortness' sake, for a function $f\in C_b(X)$, we use the same notation $f$ for the operator of multiplication by $f$ on $l^2(X)$. 

\begin{lem}\label{functions0}
Let $f\in C_b(X)$, let $\rho$ be a metric on $X\sqcup X$ compatible with $d_X$, and let $\mathcal D$ be the expanding sequence determined by $\rho$. Then 
$f\in M_\rho$ if and only if
\begin{equation}\label{e1}
\lim_{n\to\infty}\sup_{x\in X\setminus D_n}|f(x)|=0.
\end{equation}

\end{lem}
\begin{proof}
For $f\in C_b(X)$, let $f_n=f\chi_{D_n}$, where $\chi_D$ denotes the characteristic function of the set $D\subset X$. 
If $f\in M_\rho$ then for any $\varepsilon>0$ there exists $n>0$ and $m\in M_\rho$ of propagation less than $n$ such that $\|f-m\|<\varepsilon$. Let $m_0$ be the diagonal part of $m$. Then $\|f-m_0\|\leq \|f-m\|<\varepsilon$, and $m_0=g$ has propagation less than $n$. The latter means that $g(x)=0$ whenever $\rho(x,x')>n$, which is the same as $x\in X\setminus D_n$, so $|f(x)|=|f(x)-g(x)|<\varepsilon$ for any $x\in X\setminus D_n$. 

If (\ref{e1}) holds then $g_n=f\chi_{D_n}$ has propagation $\leq n$, hence $f$ lies in the closure of finite propagation diagonal operators, i.e. $f\in M_\rho$.

\end{proof}

\begin{cor}\label{functions1}
The following conditions are equivalent:
\begin{enumerate}
\item
the algebra $C_0(X,\mathcal D)$ is unital;
\item
$D_n=X$ for some $n\in\mathbb N$;
\item
$M_\rho=C_u^*(X)$.
\end{enumerate}

\end{cor}
\begin{proof}
Equivalence of (1) and (2) follows from (\ref{e1}). If (2) holds then $\rho(x,x')\leq n$ for any $x$, hence $|\rho(x,y)-\rho(x,y')|\leq n$ for any $x,y\in X$, i.e. finite propagation with respect to $\rho$ is the same as finite propagation with respect to $\rho|_X$, hence $M_\rho=C_u^*(X)$. Finally, if $M_\rho=C_u^*(X)$ then $1\in M_\rho$, so (3) implies (1).

\end{proof}

We shall write $C_0(X,\mathcal D)\widetilde{\phantom{a}}$ for $C_0(X,\mathcal D)$ if it satisfies the conditions of Corollary \ref{functions1} and for the unitalization of $C_0(X,\mathcal D)$ otherwise.

\begin{lem}\label{functions2}
Suppose that $X$ is of bounded geometry, $f\in C_b(X)$. Let $\rho$ be a metric on $X\sqcup X$ compatible with $d_X$, and let $\mathcal D$ be the expanding sequence determined by $\rho$. Then
$[a,f]\in M_\rho$ for any $a\in C_u^*(X)$ if and only if
\begin{equation}\label{e2}
\lim_{n\to\infty}\sup_{x\in X\setminus D_n}\var_{x,r}f=0
\end{equation}
for any $r>0$.

\end{lem}
\begin{proof}
First, suppose that $[a,f]\in M_\rho$ for any $a\in C_u^*(X)$. Fix $r>0$ and define $a_r\in C_u^*(X)$ by setting $(a_r)_{xy}=1$ when $\rho(x,y)\leq r$, and $(a_r)_{xy}=0$ otherwise. This is a bounded operator due to bounded geometry of $X$. Let $x\in X\setminus D_n$, which means that $\rho(x,x')>n$. As $[a_r,f]\in M$, for any $k\in\mathbb N$ there exist $l_k>0$ and $m\in M$ of propagation less than $l_k$ such that $\|[a_r,f]-m\|<\frac{1}{k}$. The latter means, in particular, that $|f(x)-f(y)-m_{xy}|<\frac{1}{k}$ for any $x,y\in X$ such that $\rho(x,y)\leq r$. 

By the triangle inequality, $\rho(x,y')\geq \rho(x,x')-\rho(x,y)>n-r$. If $n>r+l_k$ then $\rho(x,y')>l_k$, therefore, $m_{x,y}=0$, and $|f(x)-f(y)|<\frac{1}{k}$, hence $\var_{x,r}f<\frac{1}{k}$ for $x\in X\setminus D_n$ when $n$ is sufficiently great. 

Now, suppose that $\lim_{n\to\infty}\sup_{x\in X\setminus D_n}\var_{x,r}f=0$ for any $r>0$. Fix $r>0$, and let $a\in C_u^*(X)$ has propagation $\leq r$ with respect to $\rho$. Let $P_n$ denote the projection onto $l^2(D_n)\subset l^2(X)$.  Let $c=P_n[a,f](1-P_n)$. Then $c_{xy}=0$ for $x\in D_n$, so 
$$
|c_{xy}|=|a_{xy}(f(x)-f(y))|\leq \|a\|\var_{x,r}f\leq \|a\|\sup_{x\in X\setminus D_n}\var_{x,r}f 
$$
for any $x,y\in X$.
Similarly we can estimate elements of the two other corners, $(1-P_n)[a,f]P_n$ and $(1-P_n)[a,f](1-P_n)$. Let $b=[a,f]-P_n[a,f]P_n$, then
$$
|b_{xy}|\leq \|a\|\sup_{x\in X\setminus D_n}\var_{x,r}f 
$$
for any $x,y\in X$.

As $X$ has bounded geometry, the number of points in each ball of radius $r$ is bounded, let $N_r$ be this bound. For any $b\in C_u^*(X)$ of propagation $\leq r$ we have $\|b\|\leq N_r\sup_{x,y\in X}|b_{xy}|$. Hence, 
$$
\|[a,f]-P_n[a,f]P_n\|=\|b\|\leq N_r\|a\|\sup_{x\in X\setminus D_n}\var_{x,r}f,
$$
so for any $\varepsilon>0$ we can find $n_0\in\mathbb N$ such that $\|[a,f]-P_n[a,f]P_n\|<\varepsilon$ for any $n\geq n_0$.

It remains to show that $P_n[a,f]P_n\in M$. If $(P_n[a,f]P_n)_{xy}\neq 0$ then $x\in D_n$ (hence $\rho(x,x')\leq n$), and $\rho(x,y)\leq r$, therefore $\rho(x,y')\leq \rho(x,x')+\rho(x,y)\leq n+r$, thus $P_n[a,f]P_n$ has propagation less than $n+r$, hence $[a,f]\in M$.      

\end{proof}

\section{Examples}

In this section we give some examples of metrics on $X\sqcup X$ compatible with the metric $d_X$ on $X$, and related bimodules, expanding sequences and algebras. Recall that in order to distinguish points in the two copies of $X$ we write $x$ for the point $x$ in the first copy of $X$ and $x'$ for the same point in the second copy of $X$.

Let $A,B\subset X$, and let $\alpha:A\to B$ be an isometric bijection. Then we can define a metric $\rho^{A,\alpha,B}$ on $X\sqcup X$ by
$$
\rho^{A,\alpha,B}(x,y)=\rho^{A,\alpha,B}(x',y')=d_X(x,y); 
$$
$$
\rho^{A,\alpha,B}(x,y')=\inf_{z\in A}[d_X(x,z)+1+d_X(\alpha(z),y)],
$$
$x,y\in X$.
%If $A=B=X$ then we abbreviate $\rho^{A,\alpha,B}$ to $\rho^\alpha$, where $\alpha$ is an isometry of $X$.
If $B=A$ and $\alpha=\id_A$ then we abbreviate $\rho^{A,\id,A}$ to $\rho^A$. There are two extreme cases for metrics of the form $\rho^A$. 

If $A=X$ then $M_{\rho^X}=C_u^*(X)$, the corresponding expanding sequence is constant, $D_n=X$ for any $n\in\mathbb N$. Both algebras, $C_0(X,\mathcal D)$ and $C_h(X,\mathcal D)$, coincide with the whole $C_b(X)$.   

If $A=\{x_0\}$ for some point $x_0\in X$ then $M_{\rho^{x_0}}=\mathbb K(l^2(X))$, the corresponding expanding sequence $D_n=B_{\frac{n-1}{2}}(x_0)$ is the sequence of balls of radii $\frac{n-1}{2}$ centered at $x_0$. The algebra $C_0(X,\mathcal D)$ is the usual $C_0(X)$, while the algebra $C_h(X,\mathcal D)$ is the Higson algebra $C_h(X)$ of functions with variation vanishing at infinity.

The set of all uniform Roe bimodules over $C_u^*(X)$ is partially ordered by inclusion. If $M_{\rho_1}\subset M_{\rho_2}$, and $\partial D_i$ corresponds to $\rho_i$, $i=1,2$, then it is easy to see that $C_h(X,\mathcal D_2)/C_0(X,\mathcal D_2)\subset C_h(X,\mathcal D_1)/C_0(X,\mathcal D_1)$. The uniform Roe bimodule $\mathbb K(l^2(X))$ is the minimal one. 

Given an expanding sequence $\mathcal E=\{E_n\}_{n\in\mathbb N}$ of non-empty sets, we can define a metric $\rho^\mathcal E$ compatible with $d_X$ by setting $\rho^\mathcal E|_{X}=d_X$ and 
$$
\rho^\mathcal E(x,y')=\inf_{n\in\mathbb N}\inf_{z\in E_n}[d_X(x,z)+n+d_X(z,y)].
$$
It was shown in \cite{M2} that this is a metric and that its expanding sequence is equivalent to $\mathcal E$.

\section{Invariant submodules in uniform Roe bimodules}

\begin{prop}\label{HH1}
$HH^0(C_u^*(X),M_\rho)\cong\left\lbrace\begin{array}{cl}\mathbb C&\mbox{if\ }M_\rho=C_u^*(X),\\0&\mbox{otherwise.}\end{array}\right.$

\end{prop}
\begin{proof}
By definition, $HH^0(A,M)=\{m\in M:am-ma=0\mbox{\ for\ any\ }a\in A\}$. As $C_u^*(X)$ contains all compact operators,  $am=ma$ for any $a\in\mathbb K(l^2(X))$ implies that $m$ is scalar. By Corollary \ref{functions1}, only $M=C_u^*(X)$ contains non-zero scalars.

\end{proof}

\section{Derivations of uniform Roe algebras with values in uniform Roe bimodules}

In this section we calculate $HH^1(C_u^*(X),M_\rho)$ for metric space $X$ of bounded geometry using its identification with the linear space of outer derivations of $C_u^*(X)$ with coefficients in $M_\rho$.

\begin{thm}\label{mainThm}
Let $(X,d_X)$ be a discrete metric space of bounded geometry, let $\rho$ be a metric on $X\sqcup X$ compatible with $d_X$, and let $\mathcal D$ be the corresponding expanding sequence. Then $\Out(C_u^*(X),M_\rho)\cong C_h(X,\mathcal D)/C_0(X,\mathcal D)\widetilde{\phantom{a}}$.

\end{thm}

\begin{proof}

By Theorem \ref{Hel}, any derivation of $C_u^*(X)$ with values in $\mathbb B(l^2(X))$ is inner. This means that if $d:C_u^*(X)\to M_\rho\subset\mathbb B(l^2(X))$ is a bounded derivation then there exists $b\in\mathbb B(l^2))$ such that $d(a)=[a,b]$ for any $a\in C_u^*(X)$.

Let $b=b_0+b_1$, where $b_0$ (resp., $b_1$) is the diagonal (resp., off-diagonal) part of $b$ with respect to the standard basis of $l^2(X)$. Then $d=d_0+d_1$, where $d_0(\cdot)=[\cdot,b_0]$, $d_1(\cdot)=[\cdot,b_1]$.

To deal with the derivation $d_1$, we follow the argument from \cite{Lorentz-Willett}. Restrict $d_1$ to the commutative $C^*$-algebra $C_b(X)$ of diagonal operators in $C_u^*(X)$. We claim that if $[a,b_1]\in M_\rho$ for any $a\in C_b(X)$ then $b_1\in M_\rho$.

Let $\mathcal U$ be the unitary group of $C_b(X)$, equipped with the discrete topology. As $\mathcal U$ is abelian, it is amenable, and so we may fix a right-invariant mean on $l^\infty(\mathcal U)$. This allows to build a right-invariant, contractive, linear map 
$$
l^\infty(\mathcal U; \mathbb B(l^2(X)))\to\mathbb B(l^2(X)),\qquad \varphi\mapsto\int_{\mathcal U}\varphi(u)d\mu(u). 
$$
This is applied to the function $\varphi_{b_1}:\mathcal U\to \mathbb B(l^2(X))$, $\varphi_{b_1}(u)=u^*b_1u$ to get a bounded operator $b'_1=\int_{\mathcal U}\varphi_{b_1}(u)d\mu(u)\in\mathbb B(l^2(X))$. 
Note that $(b_1)_{xx}=0$ for any $x\in X$, and that $(u^*b_1u)_{xx}=0$ for any $u\in\mathcal U$, so, by Lemma 3.3 of \cite{Lorentz-Willett}, $b'_1$ is off-diagonal. On the other hand, by Lemma 3.4 of \cite{Lorentz-Willett}, $b'_1$ lies in the commutant of $\mathcal U$, hence is diagonal. Thus, $b'_1=0$.

Let us write operators on $H=l^2(X)\oplus l^2(X)$ as 2-by-2 matrices, and apply Lemma 4.1 of \cite{Lorentz-Willett} to the uniform Roe algebra of $X\sqcup X$. 
Set $\bar u=\left(\begin{smallmatrix}u&0\\0&u\end{smallmatrix}\right)$, $\bar b_1=\left(\begin{smallmatrix}0&b_1\\0&0\end{smallmatrix}\right)$. Then 
$$
[\bar u,\bar b_1]=\left(\begin{smallmatrix}0&[u,b_1]\\0&0\end{smallmatrix}\right)\in\left(\begin{smallmatrix}0&M_\rho\\0&0\end{smallmatrix}\right). 
$$
By Lemma 4.1 of \cite{Lorentz-Willett}, for any $\varepsilon>0$ there exists $r>0$ such that for any $u\in\mathcal U$ there exists $\bar m(u)\in C_u^*(X\sqcup X)$ of propagation less than $r$ such that 
\begin{equation}\label{est}
\left\|\left(\begin{smallmatrix}0&b_1-\varphi_{b_1}(u)\\0&0\end{smallmatrix}\right)-\bar m(u)\right\|<\varepsilon. 
\end{equation}
Let $m(u)$ be the upper right corner of $\bar m(u)$. Then $m\in M_\rho$, and (\ref{est}) implies that 
\begin{equation}\label{est2}
\|b_1-\varphi_{b_1}(u)-m(u)\|<\varepsilon.
\end{equation}
Averaging (\ref{est2}) over $\mathcal U$ and taking into account that $b'_1=\int_\mathcal U \varphi_{b_1}(u)d(u)=0$, we conclude that $\|b_1-m'\|<\varepsilon$, where $m'=\int_\mathcal U m(u)d(u)\in M$, hence $b_1\in M$. In other words, the derivation $d_1$ is inner.

Thus, any bounded derivation $d$ is of the form 
\begin{equation}\label{summa}
d(\cdot)=[\cdot,f]+[\cdot,m], 
\end{equation}
$m\in M_\rho$. Since $[\cdot,f]$ is a derivation, $f\in C_h(X,\mathcal D)$ by Lemma \ref{functions2}.

Define a map
$$
j:\Der(C_u^*(X),M_\rho)\to C_h(X,\mathcal D)/C_0(X,\mathcal D)\widetilde{\phantom{a}}\quad\mbox{by}\quad j(d)=f+C_0(X,\mathcal D)\widetilde{\phantom{a}}, 
$$
where $d$ is given by (\ref{summa}). 

If $d(\cdot)=[\cdot,f]+[\cdot,m]=[\cdot,f']+[\cdot,m']$ with $f,f'\in C_h(X,\mathcal D)$ and $m,m'\in M_\rho$ then $[a,f-f'-m+m']=0$ for any $a\in C_u^*(X)$. As $C_u^*(X)$ contains all compact operators, its commutant consists only of scalars, hence, $f-f'-m+m'$ is scalar. Then $g=m-m'$ is a diagonal operator, and  $g\in M_\rho$, hence $g\in C_0(X,\mathcal D)\widetilde{\phantom{a}}$, i.e. $f-f'\in C_0(X,\mathcal D)\widetilde{\phantom{a}}$, so the map $j$ is well defined. Clearly, $j$ is surjective, and its kernel is $C_0(X,\mathcal D)\widetilde{\phantom{a}}$.  

\end{proof}

\begin{cor}
If $M_\rho\neq C_u^*(X)$ then $\Out(C_u^*(X),M_\rho)$ is non-trivial.

\end{cor}
\begin{proof}
By Corollary \ref{functions1}, for every $n\in\mathbb N$ there exists $x_n\in X$ such that $\rho(x_n,x'_n)>2^{n+2}$. Without loss of generality, we may assume also that $d_X(x_n,\{x_1,\ldots,x_{n-1}\})>2^{n+1}$. Let $B_{2^n}(x_n)$ be the ball of radius $2^n$ centered at $x_n$. Note that these balls do not intersect. Define a function $f\in C_b(X)$ by 
$$
f(x)=\left\lbrace\begin{array}{cl}\frac{2^n-d_X(x,x_n)}{2^n}&\mbox{if\ }x\in B_{2^n}(x_n)\mbox{\ for\ some\ }n;\\0&\mbox{otherwise.}\end{array}\right.
$$
Clearly, $f\in C_h(X,\mathcal D)\setminus C_0(X,\mathcal D)$.

\end{proof}

\section{A remark on higher Hochschild cohomology}

A similar argument can show that the odd Hochschild cohomology is non-trivial.

\begin{thm}
If $M_\rho\neq C_u^*(X)$ then $HH^{2k+1}(C_u^*(X),M_\rho)$ contains a copy of $C_h(X,\mathcal D)/C_0(X,\mathcal D)\widetilde{\phantom{a}}$ for any $k\in\mathbb N$.

\end{thm}
\begin{proof}
Let $\psi\in C^{2k}(C_u^*(X),\mathbb B(l^2(X)))$ be given by $\psi(a_1,\ldots,a_{2k})=a_1\cdots a_{2k}f$, where $f\in C_h(X,\mathcal D)$, $a_1,\ldots,a_{2k+1}\in C_u^*(X)$. Then we have 
$$
\phi(a_1,\ldots,a_{2k+1})=\delta_{2k}(\psi)(a_1,\ldots,a_{2k+1})=\pm a_1\cdots a_{2k}[a_{2k+1},f].
$$
Clearly, $\phi\in C^{2k+1}(C_u^*(X),M_\rho)$, and $\delta_{2k+1}(\phi)=0$, hence $[\phi]\in HH^{2k+1}(C_u^*(X),M_\rho)$. It is not zero unless $f\in C_0(X,\mathcal D)\widetilde{\phantom{a}}$, as $\psi$ takes values outside $M_\rho$.  

\end{proof}

The even case is more difficult, and we cannot evaluate even the second cohomology, but we shall give some calculations for the case of the minimal uniform Roe bimodule, which is $\mathbb K(l^2(X))$. For shortness' sake we write $\mathbb K$ for $\mathbb K(l^2(X))$.

Recall that if we complete the algebraic tensor product $\mathbb K\odot_\mathbb K\mathbb K$ with respect to an appropriate norm then the formula $a\otimes b\mapsto ab$, $a,b\in\mathbb K$, defines an isometric isomorphism $\mathbb K\odot_\mathbb K\mathbb K\to\mathbb K$. We write $\mathbb K\otimes_\mathbb K\mathbb K$ for this completion, and identify it with $\mathbb K$ via the map $a\otimes b\mapsto ab$, $a,b\in\mathbb K$.

Recall that, given two $A$-bimodules over a Banach algebra $A$, the cup product for the Hochschild cohomology is defined as $[\alpha]\smallsmile[\beta]\in HH^{p+q}(A,M\otimes_A N)$ for $[\alpha]\in HH^p(A,M)$, $[\beta]\in HH^q(A,N)$. When $p=q=1$, $\alpha:A\to M$, $\beta:A\to N$, then $(\alpha\smallsmile\beta)(a,b)=\alpha(a)\otimes\beta(b)$, $a,b\in A$. If the K\"unneth theorem would hold in this situation, non-triviality of the first Hochschild cohomology would imply non-trivility of the second one. We shall show that this is not the case.  
\begin{prop} 
The map
$$
\smallsmile:HH^1(C_u^*(X),\mathbb K)\times HH^1(C_u^*(X),\mathbb K)\to HH^2(C_u^*(X),\mathbb K)
$$ 
is zero.

\end{prop}
\begin{proof}
If $\alpha$, $\beta$ are derivations of $C_u^*(X)$ with coefficients in $\mathbb K$ then, by Theorem \ref{mainThm}, $\alpha(\cdot)=[\cdot,f]$, $\beta(\cdot)=[\cdot,g]$ for some $f,g\in C_h(X)$. Then $(\alpha\smallsmile\beta)(a,b)=[a,f][b,g]$.
But $\alpha\smallsmile\beta=\delta_1(\psi)$, where $\psi(a)=[a,f]g$. As $[a,f]\in\mathbb K$, $\psi(a)\in\mathbb K$  for any $a\in C_u^*(X)$, thus $[\alpha]\smallsmile[\beta]=0$.

\end{proof}

\end{document}